\documentclass{amsart}
\usepackage{amssymb,comment}

\title[Geometric constructions of  representations of $GL_n$]{Lectures on geometric constructions of the irreducible representations of $GL_n$}

\author{Joel Kamnitzer}

\newcommand{\gl}{\mathfrak{gl}}
\newcommand{\Cx}{{\mathbb{C}^\times}}
\newcommand{\Z}{\mathbb{Z}}
\newcommand{\N}{\mathbb{N}}
\newcommand{\Sym}{\mathrm{Sym}}
\newcommand{\C}{\mathbb{C}}
\newcommand{\tf}{\mathfrak{h}}
\newcommand{\End}{\mathrm{End}}
\newcommand{\lc}{\mathcal{L}} 
\newcommand{\pb}{\mathbb{P}}
\newcommand{\Fl}{\mathrm{Fl}}
\newcommand{\cN}{\mathcal{N}}
\newcommand{\Gr}{\mathrm{Gr}}
\newcommand{\ogl}{\overline{\Gr^\lambda}}
\newcommand{\het}{\mathrm{ht}}

\newcommand{\perv}{\mathcal{P}}
\newcommand{\rep}{\mathrm{Rep}}
\DeclareMathOperator{\Hom}{Hom}
\newcommand{\tp}{{\mathrm{top}}}

\newtheorem{Theorem}{Theorem}[section]
\newtheorem{Proposition}[Theorem]{Proposition} 
\newtheorem{Lemma}[Theorem]{Lemma}

\newtheorem{Corollary}[Theorem]{Corollary}

\theoremstyle{definition}

\newtheorem{Example}[Theorem]{Example}
\newtheorem{Exercise}{Exercise}[section]

\begin{document}

\begin{abstract}
We give an exposition of three geometric constructions of the irreducible representations of $ GL_n$.  In particular, we discuss Borel-Weil theory, the Ginzburg construction, and the geometric Satake construction.  We also explain how to deduce the Ginzburg construction from the geometric Satake equivalence.  These are lecture notes for a lecture series at the Summer School on Geometric Representation Theory and Extended Affine Lie Algebras held at University of Ottawa in June 2009.
\end{abstract}

\maketitle
\tableofcontents

\section{Introduction} \label{J:se1}
In these notes, we will give three different geometric constructions of the irreducible representations of $ GL_n$.  

First, we discuss the Borel-Weil theorem, which uses line bundles on flag varieties.  The flag variety is a projective variety which parametrizes complete flags in a vector space.  The flag variety has an action of $ GL_n $ and this group also acts on the space of algebraic sections of natural line bundles.  These spaces of sections are the irreducible representation.

Second, we give Ginzburg's construction \cite{Gbook} of an action on the cohomology of $n$-step Springer fibres.  A Springer fibre parametrizes flags in a vector space which are invariant under the action of a nilpotent linear operator.  In this construction, we get an action of the generators of the Lie algebra of $ GL_n $ by convolution.

Finally, we describe the geometric Satake correspondence \cite{L,Gsat,MV}, which is a cornerstone of the geometric Langlands program.  In this approach, $ GL_n $ acts on the intersection homology of subvarieties of the affine Grassmannian -- which parametrizes certain infinite dimension subspaces of an infinite dimensional vector space.  

All three of these constructions generalize to arbitrary reductive groups.  The Borel-Weil construction works for the flag variety of any reductive group.  Ginzburg's construction has been generalized by Nakajima \cite{N} using quiver varieties.  Finally, the geometric Satake construction generalizes using the affine Grassmannian of the Langlands dual group \cite{Gsat,MV}.

In these lectures, we have chosen to focus on the case of $ GL_n$ for three reasons.  First, we can discuss the results without using too much of the terminology of reductive groups and root systems.  Second, flag varieties are much more concrete in this case.  Third, we are able to avoid introducing quiver varieties and the Langlands dual group.

The outline of these notes is as follows.  In the first section we give an overview of the basic facts concerning the representation theory of $ GL_n$.  In the second, third, and fourth sections we give the Borel-Weil, Ginzburg, and geometric Satake constructions.  In the final section we explain how to deduce the Ginzburg construction from the geometric Satake construction.  In general, we have endevoured to keep the exposition fairly elementary.  In particular, we have avoided using perverse sheaves until the final section.

I would like to thank all the participants of the summer school for their questions and comments.  I especially thank Bruce Fontaine and Lucy Zhang for reading an early draft of these notes.  I would also like to thank Erhard Neher and Alistair Savage for organizing this summer school.

\section{Representation theory of $GL_n$}\label{J:se2}
\subsection{Preliminaries and examples}\label{J:se3}
We will begin by reviewing the representation theory of the group $GL_n = GL_n(\mathbb{C})$.  First, note that we will consider algebraic (=holomorphic) representations of $ GL_n$.  Namely, a representation of $ GL_n $ is a finite-dimensional complex vector space $V $ along with a group homomorphism $ GL_n \rightarrow GL(V) $ which is also required be a morphism of algebraic varieties.

This is equivalent to considering smooth representations of its maximal compact subgroup $U(n)$.  This is also equivalent to studying finite-dimensional integral weight representations of its Lie algebra $ \gl_n$.  

The material of this section is standard.  Good references are \cite{B}, \cite{H}, and \cite{FH}.

\begin{Example} \label{J:ex1}
The simplest representation of $ GL_n $ is the ``standard representation'' which is the action of $GL_n $ on $\C^n$ (so $ V= \C^n $ and the map $ GL_n \rightarrow GL(V) $ is the identity).
\end{Example}

\begin{Example} \label{J:ex2}
More generally, we can also consider actions of $ GL_n $ on vector spaces built out of $ \C^n$.  The first examples are the actions of $ GL_n $ on the $k$th symmetric powers, $\Sym^k \C^n $, and $k$th exterior powers, $ \Lambda^k \C^n $.  The action of $ GL_n $ on $\Sym^k \C^n $ is given by
\begin{equation*}
g \cdot v_1 \cdots v_k = (g \cdot v_1) \cdots (g \cdot v_k)
\end{equation*}
and similarly for $ \Lambda^k\C^n$.
\end{Example}

\begin{Exercise} \label{J:ex3}
Identify $ \Sym^k \C^n $ with the vector space of homogeneous polynomials of degree $ k $ in $ n $ variables.  Consider the case $ n = 2$.  Describe the action of an invertible matrix $ \left[  \begin{smallmatrix} a b \\ c d \end{smallmatrix}\right]  $ on a homogeneous polynomial $p(x,y) $.
\end{Exercise}

\begin{Example} \label{J:ex4}
As a special case of the previous example, we may consider $ V = \Lambda^n \C^n $.  This is a 1-dimensional vector space.  Let $ e_1, \dots, e_n $ be the standard basis for $ \C^n $.  Then $ e_1 \wedge \cdots \wedge e_n $ is a basis for $ \Lambda^n \C^n $.

With respect to this basis, the action of $ GL_n $ becomes a group homomorphism $ GL_n \rightarrow GL_1 = \Cx$.  This homomorphism is the map taking a matrix to its determinant.  Hence we will refer to $\Lambda^n \C^n $ as the \textbf{determinant representation}.
\end{Example}

\subsection{Representations of tori} \label{J:se4}
Our goal now is to classify all representations of $ GL_n$.  

Let us begin with the case of $ GL_1 = \Cx$.  A one dimensional representation of $\Cx $ is given by a (algebraic) group homomorphism $ \Cx \rightarrow \Cx $.  All of these are of the form $ z \mapsto z^n $ for some integer $n$.  

More generally, let $ V$ be a representation of $ \Cx$.  We write 
\begin{equation*}
V_n = \{ v \in V: z \cdot v = z^n v, \text{ for all } z \in \C^\times \} 
\end{equation*}
for the subspace of vectors of weight $n$.  We have the following result.

\begin{Proposition} \label{J:thCxrep}
With the above notation, we have the direct sum decomposition $V = \oplus_{n \in \Z} V_n$.
\end{Proposition} 

So we have written $ V$ as a direct sum of eigenspaces for the elements of $ \Cx$. This proposition captures two important facts about the representation theory of $ \C^\times$.  First, all representations are semisimple; meaning that they can be written as a direct sum of irreducible subrepresentations.  Second, all irreducible representations of $ \Cx $ are $1$-dimensional (this holds since $\Cx $ is abelian).  The terms $ V_n $ appearing the direct sum decomposition in Proposition \ref{J:thCxrep} are the \textbf{isotypic} components of the representation $ V $.  By definition, $V_n $ is the set of all vectors in $ V $ which lie in a subrepresentation isormorphic to the one-dimensional representation labelled by $ n$.

More generally, let us consider a group $ T = (\Cx)^n $.  Such a group is called a \textbf{torus}.  As before, we begin by studying the one dimensional representations of this group.  These are all of the form
\begin{equation*}
z = (z_1, \dots, z_n) \mapsto z^\mu := z_1^{\mu_1} \cdots z_n^{\mu_n}
\end{equation*}
for some sequence of integers $ \mu = (\mu_1, \dots, \mu_n) \in \Z^n$.  We will call $ \mu $ a \textbf{weight} of $ T $ and call the set $P := \Z^n $ the \textbf{weight lattice} of $ T$.  

Now let $ V $ be an arbitrary representation of $ T$.  For any weight $ \mu$, we can consider the subspace of vectors of weight $ \mu$
\begin{equation*}
V_\mu := \{ v \in V : z \cdot v = z^\mu v\}
\end{equation*}

\begin{Proposition} \label{J:th1}
We have the direct sum decomposition $ V = \oplus_{\mu \in \Z^n} V_\mu $
\end{Proposition}

\subsection{Weight spaces of representations of $ GL_n$} \label{J:se5}
Within $ GL_n$, there is a large torus $ T = (\Cx)^n $ consisting of the invertible diagonal matrices.  We will study representations of $ GL_n $ by restricting them to representations of this maximal torus.

Let $ V $ be a representation of $ GL_n$.  Then it is also a representation of the maximal torus $ T $ and so we get a weight decomposition $ V = \oplus_\mu V_\mu $.  We can record some numerical information about the representation by recording the dimensions of these weight spaces in the function
\begin{equation*}
\chi_V : \Z^n \rightarrow \N, \ \mu \mapsto \dim(V_\mu).
\end{equation*}
This function is called the \textbf{character} of the representation. 

Here is the first fundamental result about the representation theory of $ GL_n$.

\begin{Theorem} \label{J:th2}
\begin{enumerate}
\item Consider the action of the symmetric group $ S_n $ on $\Z^n $ given by $$ w (\mu_1, \dots, \mu_n) = (\mu_{w(1)}, \dots, \mu_{w(n)}). $$  If $ \chi_V $ is a character, then $ \chi_V (\mu) = \chi_V(w \mu) $ for any $ w \in S_n$.
\item If $ V, W $ are two representations and $ \chi_V = \chi_W$, then $V \cong W$.
\end{enumerate}
\end{Theorem}

Thus the character completely determines the representation.

\begin{Exercise} \label{J:ex5}
Prove (i).  Hint: consider the  inclusion $S_n \hookrightarrow GL_n $ which takes permutations to permutation matrices.  
\end{Exercise}

\begin{Example} \label{J:ex6}
Consider the representation $ \Sym^3 \C^2$.  It has a basis given by $$ e_1^3, e_1^2 e_2, e_1 e_2^2, e_2^3. $$
 Note that if $z = \left[ \begin{smallmatrix} z_1 &0 \\ 0 &z_2 \end{smallmatrix} \right]$, then $$ z \cdot e_1^a e_2^b = (z_1 e_1)^a (z_2e_2)^b = z_1^a z_2^b e_1^a e_2^b $$
and hence $e_1^a e_2^b $ has weight $ (a,b) $.  Hence the non-zero weight spaces of $ \Sym^3 \C^2 $ are $ (3,0), (2,1), (1,2), (0,3) $ and each weight space has dimension 1.
\end{Example}

\begin{Example} \label{J:ex7}
More generally consider the representation $ \Sym^k \C^n $.  It has a basis given by monomials $ e_1^{a_1} \dots e_n^{a_n}$ of total degree $ k$.  Such a monomial has weight $ (a_1, \dots, a_n) $.  Hence the character is the function 
\begin{equation*}
 \chi(\mu) = 
 \begin{cases}
 1  \text{ if }  \mu_1 + \dots + \mu_n = k \\
  0 \text{ otherwise.}
\end{cases}
\end{equation*}
In particular, the weights of this representation form a simplex.
\end{Example}

\begin{Exercise} \label{J:ex8}
Find the character of the representation $ \Lambda^k \C^n $.
\end{Exercise}

\subsection{Irreducible representations of $ GL_n$} \label{J:se6}
A \textbf{subrepresentation} $ W\subset V $ of a representation of $GL_n$ is a subspace of $ W \subset V $ which is invariant under the action of $ GL_n $.

A representation $ V$ of $ GL_n $ is called \textbf{irreducible} if it has no subresentations (other than $ 0, V$).  We have the following semisimplicity result.

\begin{Theorem} \label{J:th3}
Every representation $ V $ can be written as a direct sum of irreducible subresentations.
\end{Theorem}

There are a number of proofs of this theorem.  One approach is to use the existence of a Haar measure on the maximal compact subgroup $ U(n) $.

\begin{Example} \label{J:ex9}
Consider the representation of $ GL_n $ on $\C^n \otimes \C^n $.  We have a direct sum decomposition into subrepresentations $ \C^n \otimes \C^n = \Lambda^2 \C^n \oplus \Sym^2 \C^n$.  Later we will see that $\Lambda^2 \C^n $ and $ \Sym^2 \C^n $ are irreducible representations.
\end{Example}

Now we would like to describe the irreducible representations of $ GL_n$.  For this it will be convenient to introduce a partial order on $\Z^n $.   We say $ \lambda \ge \mu $ iff we can write
\begin{equation*}
\lambda - \mu = k_1 \alpha_1 + \cdots + k_{n-1} \alpha_{n-1}
\end{equation*}
for some non-negative integers $k_1, \dots, k_{n-1} $ where $ \alpha_1 = (1, -1,0, \dots, 0), \dots, \alpha_{n-1} = (0, \dots, 0, 1, -1) $.  In this situation we write $ \het(\lambda - \mu) = k_1 + \cdots + k_{n-1} $ and call this number the \textbf{height} of $ \lambda - \mu $.

It will also be convenient to introduce the subset $ P^+ = \Z^n_+ $ of \textbf{dominant weights} which consists of weakly decreasing sequences, $ \mu_1 \ge \dots \ge \mu_n$, with all $ \mu_i \in \Z$.

A representation $ V$  is said to have \textbf{highest weight} $ \lambda $ if $ V_\lambda \ne 0 $ and whenever $ V_\mu \ne 0$, then $ \mu \le \lambda$.

\begin{Example} \label{J:ex10}
Consider $ \Sym^3 \C^2 $.  As calculated in example \ref{J:ex6}, the weights are $(3,0), (2,1), (1,2), (0,3) $.  Hence this representation has highest weight $(3,0)$.
\end{Example}

Let $ N $ denote the subgroup of $ GL_n $ consisting of those upper triangular matrices with ones on the diagonals.   Let $ V $ be a representation of $ GL_n$.  A vector $ v \in V $ is  called a \textbf{highest weight vector} if it lies in $ V_\lambda$ for some $ \lambda $ and is $ N$-invariant.

Here is the final main theorem concerning the representation theory of $ GL_n$.
\begin{Theorem} \label{J:th4}
For each $ \lambda \in \Z^n_+$, there exists a representation $ V(\lambda) $ which can be characterized in one of the following two equivalent ways.
\begin{itemize}
\item It is irreducible and of highest weight $ \lambda $.  
\item It has a 1-dimensional space of highest weight vectors and these vectors have weight $ \lambda $. 
\end{itemize}
The set $\{ V(\lambda) \}_{\lambda \in \Z^n_+} $ is the complete collection of the irreducible representations of $ GL_n$.
\end{Theorem}

These $ V(\lambda) $ are the representations that we will describe geometrically. 

An interesting problem is to determine the character $ \chi_\lambda $ of $ V_\lambda$.  Using our geometric constructions, we will see that  $ \chi_\lambda(\mu) $ can be expressed as the number of components of a certain variety (see sections \ref{J:se18} and {J:se22}).  


\begin{Example} \label{J:ex11}
Take $ \lambda = (k, 0, \dots , 0) $.  Then $ V(\lambda) = \Sym^k \C^n $.  To see this note that $\Sym^k \C^n $ has a highest weight vector $e_1^k $.  This vector is of weight $ (k, 0, \dots, 0) $ and this is its only highest weight vector (up to scalar).
\end{Example} 

\begin{Exercise} \label{J:ex12}
Show that $ \Lambda^k \C^n $ has a 1-dimensional space of highest weight vectors.  Show that this vector has weight $ (1, \dots, 1, 0 \dots, 0) $.  We will write $ \omega_k $ for this weight.  Conclude that $ \Lambda^k \C^n = V({\omega_k})$.
\end{Exercise}

\begin{Example} \label{J:ex13}
Consider $ \mathfrak{sl}_3 $ the vector space of $3 \times 3$ trace 0 matrices.  We have an action of $ GL_3 $ on this space by conjugation.  This is the irreducible representation $ V(1,0,-1)$.  In particular
\begin{equation*}
\begin{bmatrix}
0 & 0 & 1 \\
0 & 0 & 0 \\
0 & 0 & 0 
\end{bmatrix}
\end{equation*}
is a highest weight vector of weight $ (1,0,-1)$.

The weights of this representation are $$ (1,0,-1), (0,1,-1), (1,-1,0), (0,0,0), (0,-1, 1), (-1, 1, 0 ), (-1, 0, 1) .$$  All of these weight spaces are 1-dimensional with the exception of the $(0,0,0) $ weight space which consists of the diagonal matrices and hence is 2-dimensional.
\end{Example}

\subsection{Determinant representation and polynomial representations} \label{J:se7}
The determinant representation described above is the irreducible representation $ V(1, \dots, 1)$.  We may also consider its dual representation $V(-1, \dots, -1)$ which is given by $ g \mapsto \det(g)^{-1} $. Tensoring with these representations moves us around the set of dominant weights in the sense that 
\begin{align*} 
V(1, \dots, 1) \otimes V(\lambda) &\cong V(\lambda_1+1, \dots, \lambda_n+1) \\ \quad V(-1, \dots, -1) \otimes V(\lambda) &\cong V(\lambda_1 - 1, \dots, \lambda_n -1).
\end{align*}

Let $ \N^n_+ = \{ \lambda \in \Z^n_+ : \lambda_i \ge 0 \text{ for all } i \} $.  We say that an irreducible representation $ V(\lambda) $ is \textbf{polynomial}, if $ \lambda \in \N^n_+ $.  More generally, a representation of $ GL_n $ is called polynomial if it is the direct sum of polynomial irreducible representations.   

Starting with polynomial irreducible representations, we can get to all irreducible representations by repeatedly tensoring with $ V(-1, \dots, -1)$.  For this reason it is enough to consider just these polynomial irreducible representations which we will do in later sections.

\subsection{The Lie algebra action} \label{J:se8}
For some of our geometric constructions, it will be useful to think about the Lie algebra $ \gl_n$ of $ GL_n$. By definition $ \gl_n $ is the Lie algebra of $ n \times n $ matrices. 
 
Let $ \tf $ denote the Lie subalgebra of diagonal matrices in $ \gl_n $.  It is an abelian Lie algebra, usually called the Cartan subalgebra.  Define the Chevalley generators $ E_i, F_i $ for $ i = 1, \dots, n-1$ by setting $ E_i $ to be the matrix with a $1$ in the $ (i,i+1) $ entry and $ 0 $ elsewhere and setting $ F_i $ to be the transpose of $ E_i $.
The Lie algebra $ \gl_n $ is generated by elements $ E_i, F_i$ and the subalgebra $ \tf $.  

Let $ V $ be a representation of $ GL_n$.  Then there exists an action of $ \gl_n $ on $V$ by differentiation.  This means that there is a Lie algebra morphism $ \gl_n \rightarrow \End(V)$.

Suppose that $ \mu $ is a weight of $ V $.  Then $ \mu $ gives us a Lie algebra morphism $ \tf \rightarrow \C $ which takes $ X=\left[\begin{smallmatrix} a_1 &  & \\ &  \ddots & \\  &  & a_n \end{smallmatrix} \right] $ to $ \mu(X) = \mu_1 a_1 + \dots + \mu_n a_n$. 

The relationship between the action of $ GL_n $ and its Lie algebra are summarized by the following result.

\begin{Proposition} \label{J:th5}
\begin{itemize}
\item If $ v $ is a vector in $ V $ of weight $ \mu$, then for $ X \in \tf $ we have that $ X \cdot v = \mu(X)v$.
\item If $ v \in V_\mu $ then $ E_i \cdot v \in V_{\mu + \alpha_i} $ and $ F_i \cdot v \in V_{\mu - \alpha_i} $.
\item If $ v \in V_\mu $ then $ E_i F_i v - F_i E_i v = \langle \mu, \alpha_i \rangle v $ where $ \langle , \rangle $ denotes the usual bilinear form (dot product) on $ \Z^n $.
\item A non-zero vector $ v \in V_\mu $ is $N $-invariant (i.e. is a highest weight vector) iff $ E_i \cdot v = 0 $ for all $ i $.
\end{itemize}
\end{Proposition}

Using this Proposition, we get a ``picture'' of the representation $ V_\lambda $.  We have a weight space decomposition, we move around the weight spaces using $ E_i, F_i $ and we can start with our highest weight vector and generate the whole representation by moving down with the $ F_i$.  This picture can be formalized through the notion of a crystal (see for example \cite{HK}).


\section{Borel-Weil theory} \label{J:se9}
The Borel-Weil construction is the most ``classical'' of the geometric constructions we will consider and hence it can be found in many sources, for example, \cite{Gbook}, section 6.1.13 or \cite{J}, section II.5.

The geometry that will be used in this section concerns line bundles on projective varieties.  We begin by recalling some general facts.

\subsection{Line bundles on projective varieties} \label{J:se10}
Let $ \pb^n $, \textbf{projective space}, denote the set of one dimensional subspaces of $ \C^{n+1} $.  It is naturally an algebraic variety.

A \textbf{projective variety} $ X$ is a subset of projective space defined by the vanishing of some homogeneous polynomial equations.  A \textbf{line bundle} $L$ on a projective variety $X $ is an algebraic variety $\pi : L \rightarrow X $ over $ X$ such that all fibres carry the structure of a one dimensional complex vector space.  Moreover we require that $ \pi $ be locally trivial in the Zariski topology, compatible with the vector space structure on each fibre.  Usually we will construct a line bundle by describing the fibre of $ \pi $ over each point in $ X $.  If this is done in a natural way, then these fibres will glue together uniquely to form a variety $ L $.

Given the line bundle $ L $, we may consider the set $\Gamma(X, L) $ of all sections, ie $ s : X \rightarrow L $ such that $ \pi \circ s = \mathrm{id} $.  The set $ \Gamma(X, L) $ forms a vector space using the vector space structure on each fibre.

\begin{Example} \label{J:ex14}
A simple example of a line bundle is the trivial line bundle $ L= X \times \C$.  In this case, $ \Gamma(X,L) $ will be the vector space of regular functions on $ X $ (usually denoted $ \mathcal{O}(X) $).  A well-known result in algebraic geometry is that the only regular functions on a connected projective variety are constant.
\end{Example}

\begin{Example} \label{J:ex15}
Take $ Y = \pb^1 $ and consider the line bundle $ L_2 $ whose fibre at a point $ W \in \pb^1 $ is $ \C^2/W $ (one should keep in mind that a point $ W \in \pb^1 $ is a 1-dimensional subspace of $ \C^2 $).

There is an isomorphism of vector spaces (and later $ GL_2 $ representations) 
\begin{equation*}
\begin{aligned}
\C^2 &\rightarrow \Gamma(\pb^1, L_2) \\
v &\mapsto [v]_W
\end{aligned}
\end{equation*}
where for $ v \in \C^2$, we let $ [v]_W $ denote its image in $ \C^2/W $.
\end{Example}

We have the following fundamental result.
\begin{Theorem} \label{J:th6}
Let $ X $ be a projective variety and $ L $ be a line bundle.  Then $ \Gamma(X,L) $ is a finite dimensional vector space.
\end{Theorem}

Now suppose that $ X $ is a projective variety with an action of an algebraic group $ G $.  An \textbf{equivariant line bundle} $ L $ is a line bundle together an action of $ G $ on $ L $ such that $ \pi $ is equivariant for these actions.  

In this case $ \Gamma(X,L) $ will be a representation of $ G $.  This action is given as follows.  Let $s \in \Gamma(X,L) $ and $ g \in G $.  Then $ (g\cdot s)(x) = g\cdot s(g^{-1} \cdot x)$.

We will construct the irreducible representations of $GL_n $ in this manner for certain line bundles on a projective variety called the flag variety.

\subsection{The case of $ \pb^1$ } \label{J:se11}
We will begin with $ X = \pb^1$.  To understand the structure of $ \pb^1 $, note that almost every line in $ \C^2 $ is the span of $ (x, 1) $ for some $ x \in \C $.  The only line not of this form is the span of $ (1, 0) $.  This allows us to think of $ \pb^1 $ as the union of $ \C $ with a ``point at infinity''.

Note there is an action of $ GL_2 $ on $ \pb^1 $ coming from the action of $ GL_2 $ on $\C^2$.

On $ \pb^1 $ we have two natural line bundles, $ L_1 $ and $ L_2$.  Recall that a point in $ \pb^1 $ is a one dimensional subspace $ W $ of $ \C^2$.  Then the fibre of $ L_1 $ at $ W $ is $ W $ and the fibre of $ L_2 $ at $ W $ is $ \C^2/W$ (as described above).  

In other words, the total space of the line bundle $ L_1 $ is the variety of pairs 
$$\{(W, w) : w \in \C^2, W \subset \C^2, \dim W = 1, \text{ and } w \in W\}. $$ 
From this description it is clear that $ L_1 $ carries an action of $ GL_2 $ compatible with its action on $ \pb^1$.  A similar analysis holds for $ L_2$.

From these line bundles $ L_1, L_2 $ we can construct more by tensor product.  In particular, for $ (a,b) \in \N^2 $ we define the line bundle $ \lc(a, b) = L_1^{\otimes b} \otimes L_2^{\otimes a} $.  

The following theorem is the first case of our construction.

\begin{Theorem} \label{J:thBWGL2}
Assume $ a \ge b $.
We have an isomorphism of $GL_2 $ representations $ \Gamma(\pb^1,\lc(a,b)) \cong Sym^{a-b} \C^2 \otimes V(1,1)^b$.
\end{Theorem}

From our description of $ V(a,0)$ (Example \ref{J:ex11}) and our observation about tensoring with the determinant representation (Section \ref{J:se7}), we see that Theorem \ref{J:thBWGL2} implies that $ \Gamma(\pb^1, \lc(a,b)) \cong V(a,b)$.

\begin{proof}
Let us assume that $ b=0$.  There is a map $\Sym^a \C^2 \rightarrow \Sym^a (\C^2/W) = (\C^2/W)^{\otimes a}$ for any $ W $.  This gives us a map $ \Sym^a \C^2 \rightarrow \Gamma(\pb^1, \lc(a,0))$.   This map is clearly injective.  Showing that it is surjective requires a bit more machinery.

The case of $ b \ne 0 $ follows immediately since $ \lc(1,1) $ is isomorphic to the trivial line bundle except with non-trivial $ GL_2 $ action ($ GL_2 $ acts by the determinant representation).
\end{proof}

\subsection{The flag variety} \label{J:se12}
Now, we will generalize this construction for $ n \ge 2 $.  For this we will need to define a variety called the flag variety.

A \textbf{flag} in $ \C^n $ is a sequence of subspaces $ 0=V_0 \subset V_1 \subset \cdots \subset V_n = \C^n $ where $ \dim V_i = i $.  

\begin{Example} \label{J:ex16}
Let $ e_1, \dots, e_n $ be the standard basis for $ \C^n $.

The standard flag is
\begin{equation*}
0 \subset \langle e_1 \rangle \subset \langle e_1, e_2 \rangle \subset \cdots \subset \C^n
\end{equation*}
and the opposite flag is
\begin{equation*}
0 \subset \langle e_n \rangle \subset \langle e_n, e_{n-1} \rangle \subset \cdots \subset \C^n.
\end{equation*}
\end{Example}

The set of all flags in $ \C^n $ forms a projective variety denoted $ \Fl(\C^n) $.  Note that $\Fl(\C^2) $ is the same thing as $ \pb^1 $.

\begin{Exercise} \label{J:ex17}
Find an embedding $ \Fl(\C^n) \hookrightarrow \pb^N $ for some $ N $.  Hint: use the Pl\"ucker embedding.
\end{Exercise}

The group $ GL_n $ acts on the flag variety $ \Fl(\C^n) $ by acting on each piece in the flag.  In fact, this action is transitive.  To see this just note that $ GL_n $ acts transitively (in fact simply transitively) on the set of all ordered bases of $ \C^n $ and there is a surjective map from ordered bases to flags.

Consider now the set of flags $\Fl(\C^n)^T$ which are invariant under the action of the maximal torus $ T $ of $ GL_n$.  

\begin{Proposition} \label{J:th6.5}
There is a bijection between $ \Fl(\C^n)^T $ and the symmetric group $ S_n$. 
\end{Proposition}

\begin{proof}
If $ V_\bullet \in \Fl(\C^n)^T $, then each $ V_i $ is invariant under the action of $ T $.  This means that each $ V_i $ is a coordinate subspace of $ \C^n$.  From this it is immediate that $ V_\bullet $ can be constructed by taking the standard basis $ e_1, \dots, e_n $ for $ \C^n $, putting it in some order and then constructing a flag by adding one basis vector at a time.  This ordering gives us an element of $ S_n$.

Conversely if $ \sigma \in S_n$, then the corresponding $ T$-fixed flag is 
\begin{equation*}
0 \subset \langle e_{\sigma(1)} \rangle \subset \langle e_{\sigma(1)}, e_{\sigma(2)} \rangle \subset \cdots \subset \C^n.
\end{equation*}
\end{proof}

The standard and opposite flags are both $ T$-fixed and correspond to the identity and longest permutations, respectively.

Now we will consider the action of the group $ N $ of upper triangular matrices with 1s on the diagonal.  

\begin{Proposition} \label{J:th7}
The $ N $ orbit through the opposite flag is dense in $ \Fl(\C^n)$
\end{Proposition}

\begin{Exercise} \label{J:ex18}
Show that in the case of $ \pb^1 $, the $N$-orbit through the opposite flag is precisely the open subset $ \C $ in $\pb^1 $.
\end{Exercise}

\begin{Exercise} \label{J:ex19}
Prove this Proposition.
\end{Exercise}

\subsection{Line bundles on the flag variety} \label{J:se13}
For each $ \lambda \in \N^n_+$, we now define a line bundle $ \lc(\lambda) $ on the flag variety.  The fibre of $ \lc(\lambda) $ at a point $V_\bullet $ is $$ (V_1/V_0)^{\otimes {\lambda_n}} \otimes \dots \otimes  (V_n/V_{n-1})^{\otimes {\lambda_1}} .$$    Notice that this is a $ GL_n $ equivariant line bundle, so we have an action of $ GL_n $ on $ \Gamma(\Fl(\C^n), \lc(\lambda)) $.  In particular, if $ p $ is a $ T $-fixed point, then $ T $ acts on the fibre $ \lc(\lambda)_p $ over $p $.  The reason for the ``backwards'' order in the definition of $ \lc(\lambda) $ is that under this definition, the action of $ T $ on the fibre over the opposite flag is by weight $ \lambda $.

\begin{Theorem} \label{J:th8}
$\Gamma(\Fl(\C^n), \lc(\lambda)) $ is the irreducible representation of $ GL_n $ of highest weight $ \lambda $.
\end{Theorem}

\begin{proof}
It suffices to show that $ \lc(\lambda) $ has exactly one $ N$-invariant section (up to scalar) and that this section has weight $ \lambda $.

Let $ s $ be an $N$-invariant section .  Let $ p \in \Fl(\C^n)$ denote the opposite flag.  Since $ s $ is $ N $-invariant, its values on the $ N $ orbit through $ p $ are determined by its value at $ p $.  Since this $ N $ orbit is dense, this means that its values on the whole flag variety is determined by its value at $ p$.  Put another way, the map $ \Gamma(\Fl(\C^n), \lc(\lambda))^N \rightarrow \lc(\lambda)_p $ taking a section to its value at $ p $, is a surjective map.  Hence the space of $ N $-invariant sections is at most one dimensional.  

To see that there actually is a non-zero $N$-invariant section requires more detailed understanding of the geometry of the flag variety and this is where we use that $ \lambda \in \N^n_+$.  (For a good exposition, see \cite{JL}, proof of Theorem 2.)

Finally note that if $ s $ is an $ N $-invariant section, then $ (t \cdot s)(p) = t^\lambda s(p) $ since $p $ is a $ T $-fixed point and the action of $ T $ on the fibre over $ p $ is by weight $ \lambda $.  Hence if $ s(p) \ne 0 $, then $ s $ has weight $\lambda$ as desired.
\end{proof}

\begin{Example} \label{J:ex20}
Let us consider the case $ \lambda = (1, 0, \cdots, 0) = \omega_1 $.  There is a map $ \Fl(\C^n) \rightarrow \pb^{n-1} $ which takes $ (V_0, \dots, V_n) $ to $ V_{n-1} $ (here we regard $ \pb^{n-1} $ as the variety of hyperplanes in $ \C^n $).  The line bundle $ \lc(\omega_1) $ is the pullback of the quotient line bundle $ \mathcal{O}(1) $ on $ \pb^{n-1}$.  Hence we get a map
\begin{equation*}
\Gamma(\pb^{n-1}, \mathcal{O}(1)) \rightarrow \Gamma(\Fl(\C^n), \lc(\omega_1)).
\end{equation*}
This map is an isomorphism and as in the proof of Theorem \ref{J:thBWGL2}, $\Gamma(\pb^{n-1}, \mathcal{O}(1)) \cong \C^n $.  Hence we see that 
\begin{equation*}
\C^n \cong \Gamma(\pb^{n-1}, \mathcal{O}(1)) \cong \Gamma(\Fl(\C^n), \lc(\omega_1)) \cong V({\omega_1}).
\end{equation*}
\end{Example}

\begin{Exercise} \label{J:ex21}
Carry out a similar analysis when $ \lambda = \omega_k $ and when  $\lambda= (k, 0, \dots, 0) $.
\end{Exercise}

\section{Ginzburg construction} \label{J:se14}
In the Borel-Weil construction, we started with a variety with an action of $ GL_n $ and obtained a representation of $ GL_n$.  In the Ginzburg (and geometric Satake) constructions we will start with a variety that does not carry an action of $ GL_n $ and so the construction of a representation of $ GL_n $ is more indirect.  Actually instead of getting an action of $ GL_n $, we will get a vector space with a weight decomposition and an action of the generators $ E_i, F_i $ of the Lie algebra $ \gl_n$.

Ginzburg's construction was developed by Ginzburg in \cite[Chapter 4]{Gbook}, building on earlier work of Beilinson-Lusztig-MacPherson \cite{BLM}.

\subsection{Partial flag varieties, nilpotent operators and Springer fibres} \label{J:se15}
Let us fix two positive integers $ n, N$.  We will now consider the variety $\Fl_n(\C^N) $ of all $n$-step partial flags in $ \C^N$
\begin{equation*}
\Fl_n(\C^N) = \{ 0 = V_0 \subseteq V_1 \subseteq \cdots \subseteq V_n = \C^N\}.
\end{equation*}
Note that we do not put any restrictions on the dimensions of $ V_i $, however it is clear that $ \Fl_n(\C^N) $ breaks into different connected pieces according to the dimensions of the $ V_i $.  

Let $ \mu \in \N^n $ with $ |\mu | = \mu_1 + \cdots + \mu_n = N$. Then we can consider the connected component $ \Fl_\mu(\C^N) $ of those partial flags in $ \C^N $ with $ \dim(V_i/V_{i-1}) = \mu_i $.  In particular when $ n = N $, then $ \Fl_{(1, \dots, 1)}(\C^n) $ is our flag variety $ \Fl(\C^n) $ from the previous section.

We now consider the cotangent bundle $T^* \Fl_n(\C^N)$ of this $n$-step flag variety.  We will think of a point of $T^* \Fl_n(\C^N) $ as a pair $ (X, V_\bullet) $ such that $ X \in \End(\C^N) $, $ V_\bullet \in \Fl_n(\C^N) $ and  $ X V_i \subset V_{i-1} $.

Notice that those $ X$ which occur in $T^* \Fl_n(\C^N)$ will necessarily have $ X^n = 0 $.   Hence there is a projection $ \pi : T^* \Fl_n(\C^N) \rightarrow \cN_n(\C^N) $ where $ \cN_n(\C^N) $ is the variety of $ N\times N $ matrices $ X $ with $ X^n = 0 $.  The fibres of this map are called $n$-\textbf{step Springer fibres}.  We will write $ \Fl_n(\C^N)^X $ for the fibre over the nilpotent $ X$.

\begin{Example} \label{J:ex22}
Let $N=3, n=3$.  Let $ X = \left[ \begin{smallmatrix} 0 & 1 & 0 \\ 0 & 0 & 0 \\ 0 & 0 &0 \end{smallmatrix} \right] $ be an operator on $ \C^3 $.  Let $ \mu = (1,1,1) $.  Then if $V_\bullet \in \Fl_\mu(\C^3)^X $, we see that either 
\begin{equation*}
0 \subset V_1 = \langle e_1 \rangle \subset V_2 \subset \C^3 \text{ or } 0 \subset V_1 \subset V_2 = \langle e_1, e_3 \rangle \subset \C^3
\end{equation*}
This shows that $\Fl_\mu(\C^3)^X $ has two irreducible components each of which is isomorphic to $ \pb^1$.  These $ \pb^1$s are glued together at the point $ 0 \subset \langle e_1 \rangle \subset \langle e_1, e_3 \rangle \subset \C^3 $.
\end{Example}

Recall from linear algebra, that every $N\times N $ matrix is similar to a Jordan form matrix.  If we just consider nilpotent linear operators (i.e. all eigenvalues are  $0 $) then there is a bijection between Jordan form matrices and decreasing sequences $ \nu_1 \ge \dots \ge \nu_m  $ of positive integers with $ |\nu| = N $ --- here the $ \nu_i $ are the sizes of the blocks.  If $ X $ is a nilpotent linear operator on $ \C^m $, then we will speak of its \textbf{Jordan type}, which will be such a sequence.

Let $ X \in \cN_n(\C^N) $ be a nilpotent operator with Jordan type $ \nu $.  Let $ \lambda $ be the conjugate of $ \nu $ --- this means that if we write $ \nu $ as a Young diagram where the $ \nu_i $ are the lengths of the rows, then $ \lambda_i $ will be the lengths of the columns.  In other words $ \lambda_i $ is the number of $j $ such that $ \nu_j \ge i $.  Since $ X^n = 0 $,  $ \nu_i \le n $ for all $i $ and this means that $\lambda_k = 0 $ for $ k > n $. In particular, $ \lambda $ can be regarded as a dominant weight for $ GL_n $.  Also note that $ |\lambda| = N $.

\begin{Exercise} \label{J:ex23}
Let $ X, \lambda$ be as above. Suppose that the Springer fibre $  \Fl_\mu(\C^N)^X $ is non-empty.  Show that for each $ k $,
$$ \mu_1 + \dots + \mu_k  \le  \lambda_1 + \cdots + \lambda_k .$$
(Hint: first show that $ \lambda_1 + \dots + \lambda_k = \dim(\ker X^k) $.)

Conclude that as weights of $ GL_n $, we have $ \mu \le \lambda $.

Show that $ \Fl_\lambda(\C^N)^X $ is a single point.
\end{Exercise}

The intuition behind the Ginzburg construction is that we will construct a representation of $ \gl_n $ on $\Fl_n(\C^N)^X $ where the weight spaces are the $ \Fl_\mu(\C^N)^X $ and where $E_i$ acts on a flag $ V_\bullet $ to give the sum of all the flags $ V'_\bullet $ with $V_i \subset V'_i $ and $ V_j = V'_j $ for $ j\ne i $.  Of course, this description does not really make sense since $ \Fl_n(\C^N)^X $ is a variety, not a vector space.  So we will need a way to get a vector space out of a variety and for this we will use homology.

\subsection{Borel-Moore homology} \label{J:se16}
For the purpose of this construction, we will need to consider \textbf{Borel-Moore homology} of certain varieties.  

For a well-behaved topological space $ Y $ (for example a complex algebraic variety), its Borel-Moore homology is the homology of the complex of locally finite, but possibly unbounded, $i$-chains in $Y $.  This means that $ H_i(Y) $ is the group of (formal $ \C$-linear combinations of) (possibly singular and unbounded) closed $i$-dimensional subspaces of $ Y $ modulo those which are boundaries of $i+1$-dimensional subspaces.

The important properties that we will need for Borel-Moore homology are 
\begin{enumerate}
\item Pullback: if $ f : X \rightarrow Y $ is a locally trivial fibre bundle with smooth fibre of dimension $ d$, then there is a pull-back $ f^* : H_i(Y) \rightarrow H_{i+d}(X) $.
\item Pushforward: if $ f : X \rightarrow Y $ is a proper map (ie the pull back of a compact set is compact), then there is a push forward $f_* : H_i(X) \rightarrow H_i(Y) $.
\item Intersection pairing: if $M $ is a smooth manifold of real dimension $ m $ and $Z, Z', Z'' $ are three closed (well-behaved) subsets such that $ Z'' $ contains $ Z \cap Z' $, then there is an intersection product $ \cap_M : H_i(Z) \times H_j(Z') \rightarrow H_{i+j-m} (Z'')$.
\item Fundamental class: if $ Y $ is a algebraic variety of complex dimension $ n $, then we have a fundamental class $[Y] \in H_{2n}(Y) $.  If $ Y $ has multiple irreducible components $ Y_1, \dots, Y_k $ then the fundamental classes of these components $ [Y_1], \dots, [Y_k] $ form a basis for $H_{2n}(Y) $.  In this circumstance we will write $ H_\tp(Y) = H_{2n}(Y) $.
\end{enumerate}

\subsection{The main construction} \label{J:se17}
Fix $ n, N $ as above.

For each $ i = 1, \dots, n-1 $  consider the following subvariety $Z_i \subset T^* \Fl_n(\C^N) \times T^* \Fl_n(\C^N) $ defined by 
\begin{equation*}
\begin{aligned}
Z_i = \{ &(X, V_\bullet, V'_\bullet) : (X, V_\bullet), (X, V'_\bullet) \in T^* \Fl_n(\C^N), \\
&V_j = V'_j \text{ for } j \ne i, \text{ and } V_i \subset V'_i \text{ with } \dim(V'_i) = \dim(V_i) + 1 \}
\end{aligned}
\end{equation*}

Let $ \pi_1, \pi_2$ denote the two projections $ T^* \Fl_n(\C^N) \times T^* \Fl_n(\C^N) \rightarrow T^* \Fl_n(\C^N) $.

For any $ X \in \cN_n(\C^N) $ we define $ E_i : H_*(\Fl_n(\C^N)^X) \rightarrow H_*(\Fl_n(\C^N)^X) $ by
\begin{equation} \label{eq:EiGinz}
\begin{aligned}
H_*(\Fl_n(\C^N)^X) \xrightarrow{\pi_1^*} &H_*(\Fl_n(\C^N)^X \times T^* \Fl_n(\C^N)) \\
&\xrightarrow{\cap_M [Z_i]} H_*(\Fl_n(\C^N)^X \times \Fl_n(\C^N)^X) \xrightarrow{{\pi_2}_*} H_*(\Fl_n(\C^N)^X)
\end{aligned}
\end{equation}
where $ M = T^* \Fl_n(\C^N) \times T^* \Fl_n(\C^N) $ is the ambient variety where the intersection takes place.

We also define $ F_i : H_*(\Fl_n(\C^N)^X) \rightarrow H_*(\Fl_n(\C^N)^X) $ in the same manner, by interchanging the roles of $ \pi_1 $ and $ \pi_2 $ in (\ref{eq:EiGinz}).

Note that by the very definition, we see that if $ (X, V, V') \in Z_i $ and $ V \in \Fl_\mu(\C^N) $ then $ V' \in \Fl_{\mu + \alpha_i}(\C^N) $, thus we see that
\begin{equation*}
Z_i = \cup_\mu {Z_i}_\mu, \ \text{ with } {Z_i}_\mu \subset T^* \Fl_\mu(\C^N) \times T^* \Fl_{\mu+ \alpha_i}(\C^N)
\end{equation*}

This shows us that $ E_i : H_*(\Fl_\mu(\C^N)^X) \rightarrow H_*(\Fl_{\mu+ \alpha_i}(\C^N)^X) $ for each $ \mu $, and vice versa for $ F_i $.

\begin{Theorem} \label{J:th9}
This defines a representation of $ GL_n $ with weight decomposition $H_*(\Fl_n(\C^N)^X) = \oplus_{\mu} H_*(\Fl_\mu(\C^N)^X) $ and action of Chevalley generators given by $ E_i, F_i $.
\end{Theorem}

\subsection{The irreducible representation} \label{J:se18}
In general this representation will be reducible.  Now let us describe an irreducible subrepresentation.

\begin{Theorem} \label{J:th10}
Consider $ H_\tp(\Fl_n(\C^N)^X) = \oplus_{\mu} H_\tp(\Fl_\mu(\C^N)^X)$.  This subspace is invariant under the action of $ \gl_n $.  Moreover it is an irreducible representation of highest weight $ \lambda $ (where $ \lambda $ is the conjugate of the Jordan type of $ X $).  In particular there is a basis for $ V(\lambda)_\mu $ labelled by irreducible components of $ \Fl_\mu(\C^N)^X$.
\end{Theorem}

Carefully examining the homological degrees in (\ref{eq:EiGinz}) shows that $ H_\tp(\Fl_n(\C^N)^X) $ is invariant under the action of $ \gl_n $ (and hence $ GL_n$).  Also by exercise \ref{J:ex23}, we see that $ \lambda $ is the highest weight of the representation $ H_\tp(\Fl_n(\C^N)^X) $.  So to complete the proof of this theorem, we would need to know that $H_\tp(\Fl_n(\C^N))^X $ is irreducible, but this is not as easy.

One nice feature of this Ginzburg construction is that it provides us with a basis of $ V(\lambda)_\mu $.  This allows us to answer combinatorial questions --- for example we can determine $\chi_\lambda(\mu) $ by counting the irreducible components of $ \Fl_\mu(\C^N)^X $.

\begin{Example} \label{J:ex24}
Consider the case $ n=2$ and $ \lambda = (k,0) $.  So $ N = k $.  Then we see that $ X = 0 $ and for $ \mu = (a,b) $ we see that $\Fl_\mu(\C^k)^X = \Gr(a, \C^k)$, the Grassmannian of $ a $ dimensional subspaces of $\C^n $.  In particular it is an irreducible variety and so we see that $ \chi_\lambda(\mu) = 1 $ (which we already knew from our description $ V_\lambda = \Sym^k \C^2 $).
\end{Example}

\begin{Exercise} \label{J:ex25}
We continue from Example \ref{J:ex22} (so $ \lambda = (2,1,0)$ and $ \mu = (1,1,1)$).  We saw that $ \Fl_\mu(\C^3)^X $  has two irreducible components.  On the another hand we already knew that $ V(\lambda)_\mu $ was 2-dimensional from Example \ref{J:ex13}.
\end{Exercise}

\begin{Exercise} \label{J:ex26}
Continue with $ \lambda = (2,1,0) $.  We have defined the operator $$ E_1 : H_\tp(\Fl_{(1,1,1)}(\C^3)^X) \rightarrow H_\tp(\Fl_{(2,0,1)}(\C^3)^X).$$  From the previous exercise, $  H_\tp(\Fl_{(1,1,1)}(\C^3)^X) $ is 2-dimensional with a basis given by the two irreducible components.  On the other hand $\Fl_{(2,0,1)}(\C^3)^X)$ is a point (as in \ref{J:ex23}).  So with respect to this basis $ E_1 $ can be written as $1 \times 2$ matrix.  Show that it is $ [2 \ 1] $.
\end{Exercise} 

\section{Geometric Satake correspondence} \label{J:se19}
The following construction is part of the geometric Satake correspondence which is due to Lusztig \cite{L}, Ginzburg \cite{Gsat}, and Mirkovic-Vilonen \cite{MV}.  

\subsection{The varieties $\Gr^\lambda $} \label{J:se20}
For each polynomial dominant weight $ \lambda \in \N^n_+ $, we will construct a variety $ \Gr^\lambda$.  

Consider the (infinite-dimensional) $\C$-vector space $ \C[z] \otimes \C^n $.  So if $ \C^n $ has basis $ e_1, \dots, e_n$, then $ \C[z]\otimes \C^n $ has a basis $ \{ z^k e_i \} $ where $ k = 0, 1,  \dots $ and $ i = 1, \dots, n $.  Define a linear operator 
\begin{equation*}
\begin{aligned}
 X : \C[z]\otimes \C^n &\rightarrow \C[z] \otimes \C^n \\
  z^k e_i &\mapsto 
  \begin{cases} z^{k-1} e_i \text{ if } k \ge 1 \\
    0 \text{ if } k = 0 
    \end{cases}
\end{aligned}
\end{equation*}

Let $ \Gr^\ge = \{ L \subset \C[z] \otimes \C^n : X L \subset L \}$ be the ind-variety of all finite-dimensional subspaces of this vector space which are $ X $ invariant.  This is called the positive part of the \textbf{affine Grassmannian}.

Let $ L \in \Gr^\ge$.  We can consider the restriction of $ X $ to $ L $.  This will be a nilpotent operator and hence it will have a Jordan type.  Note that since $ \dim ker(X) = n $, it follows that $ \dim ker(X|_L) \le n $ which implies that $ X|_L $ has at most $n$ Jordan blocks.  Hence the Jordan type of $ X|_L $ can be regarded as an element of $ \N^n_+$.

We define
\begin{equation*}
\Gr^\lambda := \{ L \subset \C[z] \otimes \C^n : XL \subset L, \text{ and } X|_L \text{ has Jordan type } \lambda \}.
\end{equation*}
It is a locally closed subset in $\Gr^\ge $ and we may take its closure to obtain $ \overline{\Gr^\lambda} $.

It is easy to see exactly which subspaces we pick up in the closure --- in the closure the Jordan type of $ X|_L $ becomes more ``special'', i.e. more like the zero matrix.  This can be expressed as follow.
\begin{Lemma}
$$\overline{\Gr^\lambda} = \bigcup_{\mu \in \N^n_+ : \mu \le \lambda} \Gr^\mu $$
\end{Lemma}

\begin{Example} \label{J:ex26'}
Suppose that $ \lambda = \omega_k = (1, \dots, 1, 0, \dots, 0) $.  Having Jordan type $ \omega_k $ means that $ X|_L $ is 0 and that $ \dim L = k $.  Hence we see that $ \Gr^\lambda $ is isomorphic to the Grassmannian of $ k $ dimensional subspaces of $ ker(X)=\C^n $.  Note that in this case, $ \Gr^\lambda = \overline{\Gr^\lambda} $.
\end{Example}

\begin{Exercise}\label{J:ex27}
Describe $ \overline{Gr^{(2,0)}}$.
\end{Exercise}

It is not immediately obvious that these varieties $ \Gr^\lambda $ are non-empty.  Let us demonstrate this now.  For each $ \mu  \in \N^n $, we define a subspace $ L_\mu \subset \C[z] \otimes \C^n $ as the subspace generated by the action of $ X $ on $ \{ z^{\mu_1 - 1} e_1, \dots, z^{\mu_n -1} e_n \} $.

\begin{Exercise} \label{J:ex28}
$L_\lambda \in \Gr^\lambda $.  More generally $ L_{w \lambda} \in \Gr^\lambda $ for all $ w \in S_n$.
\end{Exercise}

We will construct the representation $ V(\lambda) $ of $ GL_n$ using these varieties $ \ogl$.

\subsection{Intersection homology and MV cycles} \label{J:se21}
Recall that the homology of smooth manifolds obeys Poincar\'e duality.  However, the variety $ \ogl $ is usually singular.  For well-behaved compact topological spaces (for example complex projective varieties), there is an invariant, \textbf{intersection homology}, which does obey Poincar\'e duality.  It is defined by restricting the possible intersections between the chains and the singular strata.

More precisely, suppose that $ Y $ admits a stratification $ Y = Y_n \supset Y_{n-1} \supset \cdots Y_0 $, which means that the $ Y_k $ are closed subsets such that $ Y_k \smallsetminus Y_{k-1} $ is a manifold of dimension $ 2k$ (and some other conditions).  Then $ IH_*(Y) $ is the homology of the complex of $ i $-chains $ Z \subset Y $ such that $\dim(Z \cap Y_{n-k}) \le i - k $.  

We will now describe a basis for the intersection homology of $ \ogl$ using a type of Bialynicki-Birula decomposition.

Consider the action of the torus $ T=(\Cx)^n $ on $\C[z]\otimes \C^n $ (acting on the second tensor factor).  This action commutes with the operator $ X $ and hence gives us an action on the variety of $ X$-invariant subspaces.  In particular it acts on $ \ogl$.

\begin{Exercise} \label{J:ex29}
The fixed points of this $T$-action are precisely the subspaces $ L_\mu $ described above.
\end{Exercise}

Now, define a map $ \Cx \hookrightarrow (\Cx)^n $ by $ s \mapsto \left[ \begin{smallmatrix} s^{n-1} & & \\ & \ddots & \\ & & 1 \end{smallmatrix} \right]$.  Then given a point $L \in \ogl $, we can consider $ \lim_{s \rightarrow 0 } s \cdot L $.  This limit will exist since $ \ogl $ is proper.  For each $ \mu $ such that $ L_\mu \in \ogl $, we consider the subvariety
\begin{equation*}
\{ L \in \ogl : \lim_{s \rightarrow 0} s \cdot L = L_\mu \}
\end{equation*}
The components of this subvariety are called \textbf{Mirkovic-Vilonen cycles} for $ \ogl $ of weight $\mu $.

\begin{Theorem} \label{J:th11}
The set of all MV cycles for $ \ogl $ forms a basis for $ IH_*(\ogl) $.  In particular the MV cycles of weight $\mu $ contribute to $ IH_{2k}(\ogl)$ where $ k = \het(\lambda - \mu) $.
\end{Theorem}

This theorem allows us to formally define a grading on $ IH_*(\ogl) $ by the weight lattice of $ GL_n $.  Namely, we set 
\begin{equation*}
IH_\mu(\ogl) = \text{ span of MV cycles of weight $ \mu $}.
\end{equation*}

\begin{Example} \label{J:ex30}
Let us examine this theorem in the example discussed in \ref{J:ex26'}.  So $ \lambda = \omega_k $ and $ \Gr^\lambda = \Gr(k, \C^n)$.  The $ T $-fixed points in $ \Gr(k, \C^n) $ are naturally labelled by subsets $ S $ of $ \{1, \dots, n \} $ of size $ k$.  These subsets give us coordinate subspaces $ W_S \subset \C^n $ (namely $ W_S $ is the span of $ \{e_i : i \in S \} $).  For such a subset $ S $, the attracting set
\begin{equation*}
\{ W \in \Gr(k, \C^n) : \lim_{s \rightarrow 0} s \cdot W = W_S \}
\end{equation*}
is called a \textbf{Schubert cell}.  

These Schubert cells (which are irreducible and isomorphic to an affine space) give a cell decomposition of $\Gr(k, \C^n) $.  Thus there is only one MV cycle for each $ \mu $ and it is the closure of this Schubert cell (usually called a Schubert variety).  Thus in this case $ IH_*(\ogl)= H_*(Gr^\lambda) $ has a basis given by MV cycles which are labelled by $k$-element subsets of $ \{1, \dots, n\} $. 
\end{Example}

\subsection{The representation} \label{J:se22}
To define the action of $ \gl_n $ on $ IH_*(\ogl) $ we need one more geometric ingredient.  Let $ Y $ be a projective variety with a very ample line bundle $ \lc $.  The first Chern class of $ \lc$ gives a map
\begin{equation*}
c(\lc) : IH_k(Y) \rightarrow IH_{k-2}(Y)
\end{equation*}
One may think of this map as taking a $ k$-chain and intersecting with the zero set of a generic section of $ \lc$, or equivalently, intersecting with a hyperplane inside the projective embedding coming from $ \lc$.

For our variety $ \ogl$, we define a line bundle $ \lc $ whose fibre at the point $L $ is given by $ \lc_L := \det(L)^* $, where $ \det $ of a vector space denotes the top exterior power.

This $ \lc $ gives us a map $ c(\lc) : IH_k(\ogl) \rightarrow IH_{k-2}(\ogl) $.  Using our grading by the weight lattice, we may split this $ c(\lc) $ into matrix coefficients $ c(\lc)_{\mu \nu} : IH_\mu(\ogl) \rightarrow IH_\nu(\ogl) $, where $ht(\lambda - \nu) = ht(\lambda - \mu) - 1$.  A refined analysis shows that this matrix coefficient vanishes unless $ \nu = \mu + \alpha_i $ for some $ i = 1, \dots, n-1 $.  Thus, we can split $ c(\lc) $ as a direct sum $ n-1 $ of operators $ E_i $, with $$ E_i : IH_\mu(\ogl) \rightarrow IH_{\mu + \alpha_i}(\ogl) .$$

\begin{Theorem} \label{J:th12}
The vector space $ IH_*(\ogl) $ with its weight space decomposition and $ E_i $ defined as above is a representation of $ \gl_n $.  In particular it is the irreducible representation of highest weight $ \lambda $.
\end{Theorem}

One might wonder where the $ F_i $ come from.  A deep result in the topology of algebraic varieties, called the Hard Lefschetz theorem, states that for any projective variety $ Y $, there exists an operator $ c(\lc)^* $ such that the pair $c(\lc), c(\lc)^* $ generate an action of $ \mathfrak{sl}_2 $  on $ IH_*(Y) $ compatible with the grading.  To get our $ F_i $, we then split $c(\lc)^* $ into $ F_i $, just as we split our $ c(\lc) $ into $ E_i $ (except we actually have to divide by certain coeffiecients when defining the $ F_i $).

As a corollary of Theorems \ref{J:th11} and \ref{J:th12}, we find that the number of MV cycles for $ \ogl $ of weight $ \mu $ equals the dimension of the $ \mu $ weight space in $ V(\lambda)$.

\begin{Example} \label{J:ex31}
Continuing from Example \ref{J:ex30}, where we saw that $ IH_*(Gr^{\omega_k}) $ has a basis labelled by $ k $ element subset of $ \{1, \dots, n\} $.  Also $V({\omega_k}) = \Lambda^k \C^n $ has a basis labelled by $ k $ elements subsets of $\{1, \dots, n \} $, so this isomorphism is clear in that case.
\end{Example}

\begin{Exercise} \label{J:ex32}
Consider $ \overline{\Gr^{(2,1,0)}} $.  Show that there are 2 MV cycles of weight $ (1,1,1) $.
\end{Exercise}

\section{Geometric skew Howe duality} \label{J:se23}
We will discuss a ``duality'' relation between the Ginzburg and geometric Satake constructions.  In particular, this will allow us to deduce the Ginzburg construction from the geometric Satake correspondence, following ideas of Braverman, Gaitsgory, Mirkovic, and Vybornov (see \cite{BG,BGV,MVy}).  

We will begin by giving a more refined statement of the geometric Satake correspondence.

\subsection{Geometric Satake} \label{J:se24}
Let $ Y = \cup Y_\alpha  $ be a stratified space.  Here each $ Y_\alpha $ is a locally closed submanifold of $ Y $.  Also we require that the closure of any $ Y_\alpha $ is the union of various $ Y_\beta $.  To simplify the following discussion, let us assume that all of the $ Y_\alpha $ are simply-connected.

In this case, there is a category $ \perv(Y) $ of perverse sheaves on $ Y $ which are constructible with respect to this stratification.  Perverse sheaves are a subset of the set of all complexes of constructible sheaves on $ Y $ which obey certain vanishing conditions with respect to their stalks and costalks.  For more information on the definition and basic properties of perverse sheaves, see \cite{Gbook}, section 8.4.

We will need the following properties.
\begin{Proposition} \label{J:th13}
\begin{itemize}
\item The category $\perv(Y) $ is an abelian category.
\item The irreducible objects of $ \perv(Y) $ are the IC sheaves $ IC(\overline{Y_\alpha})$ of the strata.  These can be characterized as the ``simplest'' perverse sheaves which are supported on $\overline{Y_\alpha} $ and whose restriction to $ Y_\alpha $ is the (shifted) constant sheaf.
\item The push-forward of a perverse sheaf under a semi-small map is perverse (see \cite{Gbook}, section 8.9).
\end{itemize}
\end{Proposition}

Given a perverse sheaf $ \mathcal{F} $ on $ Y $, we may consider the cohomology of $ Y $ with coefficients in $ \mathcal{F}$.  In particular we have  $ H^*(Y, IC(\overline{Y_\alpha})) = IH(\overline{Y_\alpha})$.

Recall that $ \Gr^\ge = \cup_{\lambda \in \N^n_+} \Gr^\lambda$ is the space of all subspaces $ L \subset \C[z] \otimes \C^m $ such that $ XL \subset L$.  We will consider it as a stratified space with strata $ \Gr^\lambda $.  

Let $ \rep^\ge(GL_m) $ denote the category of polynomial representations of $ GL_m$.  Now we can state the geometric Satake correspondence.

\begin{Theorem} \label{J:th14}
There is an equivalence of categories
\begin{equation*}
\perv(\Gr^{\ge}) \cong \rep^\ge(GL_m)
\end{equation*}
This equivalence is compatible with the functors to the category of vector spaces (denoted $ Vect $) on both sides.  One functor is $ \mathcal{F} \mapsto H^*(\Gr^\ge, \mathcal{F})$ and the other functor is the forgetful functor.

Under this equivalence $ IC(\ogl) $ is taken to $ V(\lambda) $.
\end{Theorem}

Using the compatibility with the functors to $ Vect $, we get our previous statement that $ IH(\ogl) \cong V(\lambda) $.

The above equivalence is actually an equivalence of tensor categories with respect to a tensor structure on $ \perv(\Gr^\ge)$.  Instead of giving a full definition of this tensor structure, we will explain how it works for the IC sheaves.

Let  $\lambda^1, \dots, \lambda^n $ be a sequence of dominant weights of $ GL_m$.  Then we define the variety
\begin{equation*}
\begin{aligned}
\overline{\Gr^{\lambda^1, \dots, \lambda^n}} := \{ 0 = &L_0 \subset L_1 \subset \cdots \subset L_n \subset \C[z]\otimes \C^m : \\
 &XL_i \subset L_i \text{ and } X|_{L_i/L_{i-1}} \text{ has Jordan type } \le \lambda^i \}
 \end{aligned}
 \end{equation*}
  
 There is an obvious map $ m_{\lambda^1, \dots, \lambda^n} : \overline{\Gr^{\lambda^1, \dots, \lambda^n}}  \rightarrow \Gr $ taking $(L_0, \dots, L_n) $ to $ L_n$.  Actually its image is in $ \overline{\Gr^{\lambda^1 + \dots + \lambda^n}}$.
 
 \begin{Theorem} \label{J:th15}
 Under the geometric Satake equivalence, $m_*(IC(\overline{\Gr^{\lambda_1, \dots, \lambda_n}})) $ corresponds to $ V({\lambda^1}) \otimes \cdots \otimes V({\lambda^n})$.
 \end{Theorem}
 
 \begin{Corollary} \label{J:thgeomSatcor}
 Let $ \mu, \lambda^1, \dots, \lambda^n \in \N^m_+$.  Then
 \begin{equation*}
 \begin{aligned}
 \Hom_{GL_m}\big( V(\mu),  V({\lambda^1}) \otimes \cdots \otimes V({\lambda^n})) &\cong \Hom_{\perv(\Gr)}(IC(\overline{\Gr^\mu} \big), m_*(IC(\overline{\Gr^{\lambda_1, \dots, \lambda_n}})) \\
 &\cong H_\tp(m_{\lambda^1, \dots, \lambda^n}^{-1}(L_\mu)).
 \end{aligned}
 \end{equation*}
 \end{Corollary}
 
 For the second isomorphism, we take the stalk of $ m_*(IC(\overline{\Gr^{\lambda_1, \dots, \lambda_n}})) $ at the point $ L_\mu $ and use the fact that $ m_{\lambda^1, \dots, \lambda^n} $ is strictly semismall (for a general statement of this type, see \cite[Theorem 5.4]{Gnotes}).
 
 \subsection{Skew-Howe duality} \label{J:se25}
 Fix $ n, m $.  Consider the vector space $ \C^m \otimes \C^n$.  This vector space has commuting actions of $ GL_n, GL_m$.  
 
 Let $ V $ be a vector space with actions of two reductive groups $ G, H$.  These actions are called \textbf{Howe dual} if the actions commute and the image of $ G $ generates the $ \Hom_H(V, V) $ (ie $ U\mathfrak{g} \rightarrow \Hom_H(V,V) $ is onto) and vice versa.

\begin{Proposition} \label{J:th16}
Under this situation, $ V \cong \oplus_{i=1}^k U_i \otimes W_i $ where $ U_i $ is an irreducible representation of $ G $, $ W_i $ is an irreducible representation of $ H $.  In particular $\Hom_G(W_i, V) \cong U_i $ as $H $-representations.  
\end{Proposition}

\begin{Exercise} \label{J:ex33}
Prove this proposition. 
\end{Exercise}
 
Now we fix $N $ and consider $ \Lambda^N(\C^n \otimes \C^m)$.  This vector space has actions of $ GL_n $ and $ GL_m $ coming from their actions on $ \C^n \otimes \C^m $. In fact these actions of $ GL_n $ and $ GL_m $ are Howe dual.  We have the following result.
 
\begin{Theorem} \label{J:th17}
\begin{equation*}
\Lambda^N(\C^n \otimes \C^m) \cong \bigoplus_{\lambda \in \N_m^+ , |\lambda | =  N, \lambda_i \le n} V({\lambda^\vee}) \otimes V(\lambda)
\end{equation*}
\end{Theorem}

Here  $ \lambda^\vee $ is the conjugate of $ \lambda $ as defined in section \ref{J:se15}.  The Young diagrams of these $ \lambda $ fit inside a $n \times m $ box.

In particular this theorem shows that 
\begin{equation} \label{eq:fromSkewHowe}
\Hom_{GL_m}(V(\lambda), \Lambda^N(\C^n \otimes \C^m)) \cong V({\lambda^\vee})
\end{equation}
as representations of $GL_n $.

\begin{Example} \label{J:ex34}
Consider $ \Lambda^2(\C^2 \otimes \C^3)$.  By the theorem we have
\begin{equation*}
\Lambda^2(\C^2 \otimes \C^3) \cong V(2,1) \otimes V(2,1,0) \oplus V(3,0) \otimes V(1,1,1)
\end{equation*}
Alternatively, we can see that as $ GL_3$ representations
\begin{equation*}
\begin{aligned}
\Lambda^2(\C^2 \otimes \C^3) &\cong \Lambda^3(\C^3 \oplus \C^3) \\
&\cong \Lambda^3 \C^3 \oplus (\Lambda^2\C^3 \otimes \C^3) \oplus (\C^3 \otimes \Lambda^2\C^3) \oplus \Lambda^3 \C^3 \\
&\cong V(1,1,1) \oplus V(2,1,0) \oplus V(1,1,1) \oplus V(2,1,0) \oplus V(1,1,1) \oplus V(1,1,1).
\end{aligned}
\end{equation*}
So we see $ 4 = \dim V_{(3,0)} $ copies of $ V_{(1,1,1)} $ and $2 = \dim V_{(2,1)}$ copies of $ V_{(2,1,0)} $ as expected.
\end{Example}

\subsection{Deriving the Ginzburg construction} \label{J:se26}
We begin by noting that as $ GL_m $ representations, we have
\begin{equation*}
\Lambda^N(\C^n \otimes \C^m) \cong \Lambda^N(\C^m \oplus \cdots \oplus \C^m) 
\end{equation*}
and now we can expand out the right hand side using a ``binomial formula'' to deduce that
\begin{equation*}
\Lambda^N(\C^n \otimes \C^m) \cong \bigoplus_{\mu \in \N^n, | \mu| = N} \Lambda^{\mu_1} \C^m \otimes \cdots \otimes \Lambda^{\mu_n} \C^m 
\end{equation*}
Combining this with (\ref{eq:fromSkewHowe}), we learn that there is an isomorphism of $ GL_n $ representations
\begin{equation*}
\bigoplus_{\mu \in \N^n, |\mu| = N} \Hom_{GL_m}(V_\lambda,  \Lambda^{\mu_1} \C^m \otimes \cdots \otimes \Lambda^{\mu_n} \C^m ) \cong V({\lambda^\vee})
\end{equation*}
Moreover, we note that the direct sum decomposition on the left hand side matches the weight decomposition of the right hand side.

Now, we apply Corollary \ref{J:thgeomSatcor} to deduce that
\begin{equation*}
V({\lambda^\vee}) \cong \bigoplus_\mu H_\tp(m_{(\omega_{\mu_1}, \dots, \omega_{\mu_n})}^{-1}(L_\lambda))
\end{equation*}
(here we use that $ \Lambda^k \C^m = V({\omega_k}) $).

Let us carefully examine the last variety
\begin{equation*}
\begin{aligned}
m_{(\omega_{\mu_1}, \dots, \omega_{\mu_n})}^{-1}(&L_\lambda) = \{ 0 = L_0 \subset L_1 \subset \cdots \subset L_n = L_\lambda : \\
&X L_i \subset L_i, X|_{L_i/L_{i-1}} \text{ has Jordan type } \omega_{\mu_i} \}
\end{aligned}
\end{equation*}
However, since $ \omega_k = (1,\dots, 1, 0, \dots, 0) $, having Jordan type $ \omega_k $ is the same thing as being the zero matrix on a space of dimension $ k $.  So we see that we can rewrite the above condition as $ XL_i\subset L_{i-1} $ and $ \dim L_i/L_{i-1} = \mu_i $.  Hence we conclude that 
\begin{equation*}
 m_{(\omega_{\mu_1}, \dots, \omega_{\mu_n})}^{-1}(L_\lambda) = \Fl_\mu(L_\lambda)^X
 \end{equation*}
 
Thus using the geometric Satake correspondence and the skew Howe duality we have proven the following.
\begin{Theorem} \label{J:th18}
Identifying $ L_\lambda $ with $ \C^N$, we find that $ \bigoplus_{\mu \in \N^n, |\mu|= N} H_\tp(\Fl_\mu(\C^N)^X) $ has an action of $ GL_n $ and is isomorphic to $ V({\lambda^\vee})$.
\end{Theorem}

Since the Jordan type of $ X $ on $ L_\lambda $ is $ \lambda $, this reproves the main theorem of the Ginzburg construction.

One might ask if these two constructions are the same.  We have constructed two (possibly different) actions of $ GL_n $ on $ \bigoplus_{\mu \in \N^n, |\mu|= N} H_\tp(\Fl_\mu(\C^N)^X) $.  In the first, we defined the action of the Chevalley generators using the varieties $ Z_i $.  In the second we worked more abstractly using the geometric Satake correspondence and skew Howe duality to define an action of the entire $ GL_n $ at once.  

\begin{Exercise} \label{J:ex35}
Prove that these two action coincide.  Hint: first, we can see easily that the weight decompositions coincide.  So it remains to see that in our second version of the story, the Chevalley generators act using the varieties $ Z_i $. This can be done by expressing the action of $ E_i $ on $$ \Hom_{GL_m}(V(\lambda),  \Lambda^{\mu_1} \C^m \otimes \cdots \otimes \Lambda^{\mu_n} \C^m ) \cong V_{\lambda^\vee}(\mu)$$ as the composition with basic $ GL_m $-equivariant maps $$ \Lambda^k \C^m \otimes \Lambda^l \C^m \rightarrow \Lambda^{k-1} \C^m \otimes \C^m \otimes \Lambda^l \C^m \rightarrow \Lambda^{k-1} \C^m \otimes \Lambda^{l+1} \C^m, $$
and then expressing these basic maps using the geometric Satake correspondence.
\end{Exercise}

To conclude, in this section, we have seen a geometric incarnation of skew-Howe duality, where one group acts by the Ginzburg construction and the other group acts by the geometric Satake construction\footnote{A geometric manifestation of symmetric-Howe duality was acheived earlier by Wang \cite{W}, using the Ginzburg construction on both sides.}.

\bibliography{schoolfinal}{}
\bibliographystyle{plain}

\end{document}